\documentclass[11pt]{amsart}
\usepackage{ amsmath}
\usepackage{amssymb}
\usepackage{amsxtra}
\usepackage{enumerate}
\usepackage[mathscr]{eucal}
\usepackage[margin=3.5cm]{geometry}

\def\R{\mathbb R}

\def\SL{{\rm SL}}

\def\SL{{\rm SL}}

\def\c{\mathcal C}

\newtheorem{theorem}{Theorem}[section]
\newtheorem{lemma}[theorem]{Lemma}

\theoremstyle{definition}
\newtheorem{definition}{Definition}

\theoremstyle{remark}

\numberwithin{equation}{section}
\theoremstyle{plain}

\newcommand{\secref}[1]{Section~\ref{#1}}
\newcommand{\thmref}[1]{Theorem~\ref{#1}}

\newcommand{\eqnref}[1]{~{\textrm(\ref{#1})}}

\begin{document}
\title[On generalized J\o{}rgensen inequality in infinite dimension]{On generalized J\o{}rgensen inequality in infinite dimension}
\author[Krishnendu  Gongopadhyay]{Krishnendu Gongopadhyay}
 \address{ Indian Institute of Science Education and Research (IISER) Mohali,
Knowledge City, Sector 81, SAS Nagar, Punjab 140306, India}
\email{krishnendu@iisermohali.ac.in, krishnendug@gmail.com}

\date{\today}
\thanks{Gongopadhyay acknowledges partial support from SERB MATRICS grant MTR/2017/000355.}

\begin{abstract}
In \cite{li1}, Li has obtained an analogue of the J\o{}rgensen inequality in the infinite-dimensional   M\"obius group. We show that this inequality is strict.  
\end{abstract}
\maketitle

\section{Introduction}
The M\"obius group $M(n)$ acts by isometries on the $n$-dimensional real hyperbolic space. The J\o{}rgensen inequality is a pioneer result in the theory of discrete subgroups of M\"obius groups. The classical J\o{}rgensen inequality gives a necessary criterion to detect the discreteness of a two-generator subgroup in $M(2)$ and $M(3)$. There have been several generalization of the J\o{}rgensen inequality in higher dimensional M\"obius groups, e.g. \cite{her}, \cite{martin}, \cite{wat}. 

The Clifford algebraic formalism to M\"obius group was initiated by Ahlfors in \cite{ahl}.  In this approach the $2 \times 2$ matrices over finite dimensional Clifford algebra acts by linear fractional transformations on the $n$-sphere.  Waterman used the Clifford algebraic formalism of M\"obius groups to obtain  some J\o{}rgensen type inequalities in \cite{wat}.    Frunz\u a  initiated a framework for  infinite dimensional M\"obius group in \cite{fr}. This framework is an extension of the Clifford algebraic viewpoint  by Ahlfors. In \cite{li1, li2, li3},  Li has used this viewpoint further to obtain discreteness criteria in infinite dimension. 

 In \cite{li1}, Li has obtained an analogue of J\o{}rgensen inequality in the infinite-dimensional   M\"obius group.  The aim of this note is to show that this inequality is strict.   In \secref{prel}, we briefly recall  basic notions of the infinite dimensional theory and note down the J\o{}rgensen type inequality of Li.   In \secref{thm} we prove that Li's inequality is strict, see \thmref{thm1}.

\section{Preliminaries}\label{prel}

\subsection{Infinite dimensional Clifford group} The Clifford algebra $\c$ is the associative algebra over $\R$ generated by a countable family $\{i_k\}_{k=0}^{\infty}$ subject to the relations: 
$$i_hi_k=-i_k i_h, ~h \neq k, ~~i_k^2=-1,$$
and no others. Every element of $\c$ can be expressed as $a=\sum a_I I$, where $I=i_{k_1}i_{k_2} \ldots i_{k_p}, ~ 1 \leq k_1 < k_2 <\cdots< k_p \leq n$, $n$ is a fixed natural number depending upon $a$, $a_I \in \R$, and $\sum_I a_I^2 < \infty$. If $I=\emptyset$, then $a_I$ is the real part of $a$ and the remaining part is the `imaginary part' of $a$. In $\c$ the Euclidean norm is given as usual by
$$|a|=\sqrt{|Re(a)|^2 + ||Im(a)|^2}. $$
As in the finite-dimensional Clifford algebra, $\c$ has three special involutions, defined by the following. 

$*$: In $a \in \c$ as above, replace in each $I=i_{v_1}i_{v_2}\cdots i_{v_k}$ by $i_{v_k}\cdots i_{v_1}$. $a \mapsto a^{\ast}$ is an anti-automorphism.

$'$: Replace $i_k$ by $-i_k$ in $a$ to obtain $a'$.

The conjugate $\bar a$ of $a$ is now defined as: $\bar a=(a^{\ast})'=(a')^{\ast}$.

\medskip Elements of the following type:
$$a=a_0 + a_1 i_1 + \cdots+ a_n i_n + \cdots,$$
are called \emph{vectors}. The set of vectors is denoted by $\ell_2$. Let $\overline \ell_2=\ell_2\cup \{\infty\}$. For any $x \in \ell_2$, we have $x^{\ast}=x$ and $\bar x=x'$. Every non-zero  vector is invertible and $x^{-1}=\bar x/|x|^2$. The set of products of finitely many non-zero vectors is a multiplicative group, called Clifford group, and denoted by $\Gamma$. 

\medskip 
 A Clifford matrix $g=\begin{pmatrix} a & b \\c & d \end{pmatrix}$ over $\ell_2$ is defined as follows: 

\begin{enumerate}
\item $a, b, c, d \in \Gamma \cup \{0\}$;
\item $\Delta(g)=ad^{\ast}-bc^{\ast}=1$;
\item $ab^{\ast}, d^{\ast} b , cd^{\ast}, c^{\ast} a \in \ell_2.$
\end{enumerate}

The set of all such matrices form a group, denoted by $\SL(\Gamma)$. For $g$ as above, 
$g^{-1}=\begin{pmatrix} d^{\ast} & - b^{\ast} \\ -c^{\ast} & a^{\ast} \end{pmatrix}$. Note that $gg^{-1}=g^{-1}g=I$. 

The group ${\rm PSL}(\Gamma)=\SL(\Gamma)/\{\pm I\}$ acts on $\overline \ell_2$ by the following transformation:
$$g: x \mapsto (ax+b)(cx+d)^{-1}. $$

\subsection{Classification of elements in $\SL(\Gamma)$}
Let $f$ be in $\SL(\Gamma)$. Then

\begin{itemize} 
\item $f$ is \emph{loxodromic} if it is conjugate in $\SL(\Gamma)$ to $\begin{pmatrix} r \lambda & 0 \\ 0 & r^{-1} \lambda'\end{pmatrix}$, where $r \in \R-\{0\}$, $|r| \neq 1$, $\lambda \in \Gamma$. If $\lambda=\pm 1$, then $f$ is called \emph{hyperbolic}. 

\item $f$ is parabolic if it is conjugate in $\SL(\Gamma)$ to $\begin{pmatrix} a & b \\ 0 & a' \end{pmatrix}$, where $a, b \in \Gamma$, $|a|=1$, $b \neq 0$, and $ab=ba'$. 

\item Otherwise $f$ is elliptic. 
\end{itemize}

\begin{definition}
For $g=\begin{pmatrix} a & b \\ c & d  \end{pmatrix}$, the \emph{trace} of $g$ is defined by
$$tr(g)=a+ d^{\ast}.$$
A non-trivial element $g \in \SL(\Gamma)$ as above is called \emph{vectorial} if $b^{\ast}=b$, $c^{\ast}=c$, and $tr(g) \in \R$. 
\end{definition}
The real part of trace is a conjugacy invariant in $\SL(\Gamma)$. 
\begin{lemma} \cite{liw, li1} 
If an element $g$ in $\SL(\Gamma)$ is hyperbolic then $tr(g) \in \R$, $tr^2(g)>4$. 
\end{lemma}

\begin{definition}
A subgroup $G$ of $\SL(\Gamma)$ is called elementary if it has a finite orbit in $\ell_2$. Otherwise, $G$ is called non-elementary. 

 A subgroup $G$ of $\SL(\Gamma)$ is discrete if for a sequence $f_i \to g$ in $G$ implies that $f_i=g$ for all sufficiently large $i$. Otherwise $G$ is not discrete. 
\end{definition} 
\subsection{Li-J\o{}rgensen inequality}The following is the generalized Jo{}rgensen inequality in infinite dimensional that was given by Li in \cite{li1}. 
\begin{theorem} \cite[Theorem 3.1]{li1} \label{lit} 
Let $f, g \in \SL(\Gamma)$ be such that $f$ is hyperbolic, and $[f, g]=f g f^{-1} g^{-1}$ is vectorial. Suppose that the two-generator group $\langle f, g \rangle$ is discrete and non-elementary. Then
\begin{equation}\label{li0} |tr^2(f)-4|+|tr([f, g])-2| \geq 1. \end{equation} 
\end{theorem}

\section{Li-J\o{}rgenesen Inequality is Strict} \label{thm}
\begin{theorem}\label{thm1}
Let $f, g \in \SL(\Gamma)$ be such that $f$ is hyperbolic, and $[f, g]=f g f^{-1} g^{-1}$ is vectorial. Suppose that the two-generator group $\langle f, g \rangle$ is discrete and non-elementary. Then
\begin{equation}\label{li0} |tr^2(f)-4|+|tr([f, g])-2| > 1, \end{equation} 
where the above inequality is strict. 
\end{theorem} 
\begin{proof}
It follows from \thmref{lit} that 
$$ |tr^2(f)-4|+|tr([f, g])-2| \geq 1. $$
If possible suppose that  
\begin{equation}\label{li} |tr^2(f)-4|+|tr([f, g])-2| = 1. \end{equation} 
Up to conjugacy, we assume 
$f=\begin{pmatrix} r & 0 \\ 0 & r^{-1} \end{pmatrix}$, $r>1$. Let $g=\begin{pmatrix} a & b \\ c & d \end{pmatrix}$. Let $J(f, g)$ denote the right hand side of \eqnref{li}.

By computation it is easy to see that 
$$tr([f, g])-2=-(r-r^{-1})^2bc^{\ast};$$
$$tr^2(f)-4=(r-r^{-1})^2.$$
So,  $J(f, g)=(r-r^{-1})^2(1 + |bc^{\ast}|)=1$. 
Since $[f, g]$ is vectorial, it follows from above that $bc^{\ast}$ is a real number.

Let $g_0=g$, $g_{m+1}=g_m f g_m^{-1}$, $g_m=\begin{pmatrix} a_m & b_m \\ c_m & d_m \end{pmatrix}$.
Let $K=(r-r^{-1})^2$,  $w_m=b_m c_m^{\ast}$. 

\medskip Then by the equality in \eqnref{li} we have $K(1+|w_0|)=1$. This implies $K<1$. 
 
Now note that 
\begin{equation}\label{0} b_{m+1} c_{m+1}^{\ast}=-K(1+b_m c_m^{\ast}). b_m c_m^{\ast}.\end{equation}
 By induction, $w_m=b_m c_m^{\ast}$ is a sequence of real numbers. Also 
$$|w_{m+1} | \leq K |w_m| (1+|w_m|).$$
If possible suppose $\alpha_m=K(1+|w_{m}|)<1$ for some $m$. Then using arguments similar to the proof of \cite[Theorem 3.1]{li1}, it can be shown that $|b_{m+n}c_{m+n}^{\ast} | \leq \alpha_{m}^n |b_mc_m^{\ast}|$ and $b_{m+n} c_{m+n}^{\ast} \to 0$ as $n \to \infty$, that would give a contradiction to the assumption that  $\langle f, g \rangle$ is non-elementary. So, we must have $K(1+|w_m|) \geq 1$ for all $m$. 

Thus 
$$1 \leq K(1+|w_m|) \leq K (1+|w_{m-1}|).$$ It is given that $K(1+|w_0|)=1$. By induction, it follows that for all $m$,  
\begin{equation} \label{1} J(f, g_m)=K(1+|w_m|)=1. \end{equation}
Note from \eqnref{0} and \eqnref{1}, 
$$1-K=K|w_{m+1}|\leq K. K|w_m|. |1+w_m| \leq (1-K)K|1+w_m|\leq (1-K) K(1+|w_m|) \leq (1-K),$$
that implies, 
\begin{equation}\label{2} K|1+w_m|=1. \end{equation}

\medskip Observe that
 $$|tr([f, g])-2 + 4 - tr^2(f)|=K|1+bc^{\ast}|=1=|tr([f, g])-2|+|4-tr^2(f)|. $$
Since $4-tr^2(f)<0$, this implies $w_0>0$. Hence by induction from \eqnref{1} and \eqnref{2},  $w_m>0$ for all $m$. Thus, we have from \eqnref{1}, $K={1}/(1+w_m)$. In particular, $w_m=w_{m+1}$. Now,  from \eqnref{0}, we have 
$K(1+w_m)=-1$, i.e. $K=-1/(1+w_m)$. This is a contradiction. Hence the inequality \label{li0} must be strict. 
\end{proof}

\end{document}